\documentclass[reqno]{amsart}

\usepackage{latexsym}
\usepackage{amssymb}

\usepackage{amsthm}
\usepackage{hyperref}
\usepackage{doi,url}

\newtheorem{theorem}{Theorem}[section]

\newtheorem{corollary}[theorem]{Corollary}

\theoremstyle{definition}

\numberwithin{equation}{section}

\newcommand{\eq}[1]{\eqref{#1}}

\DeclareMathOperator{\tr}{tr}
\DeclareMathOperator{\Ran}{Ran}
\DeclareMathOperator{\dist}{dist}

\newcommand\R{\mathbb R}

\newcommand{\abs}[1]{\left\lvert #1 \right\rvert}

\newcommand{\pa}[1]{\left( #1 \right)}

\newcommand\beq{\begin{equation}}
\newcommand\eeq{\end{equation}}

\newcommand{\qtx}[1]{\quad\text{#1}\quad}
\newcommand{\sqtx}[1]{\;\text{#1}\;}

\begin{document}
\title[Local behavior of solutions  and bounds on the density of states]{Local behavior of solutions of the stationary Schr\" odinger equation with singular potentials and bounds on the density of states of Schr\"{o}dinger operators}
\author{Abel Klein \and C.S. Sidney Tsang}
\address{University of California, Irvine,
Department of Mathematics,
Irvine, CA 92697-3875,  USA}
 \email[Klein]{aklein@uci.edu}
  \email[Tsang]{tsangcs@uci.edu}

  \thanks{A.K. and C.S.S.T. were  supported  by the NSF under grant DMS-1301641.}


\begin{abstract}
  We study the local behavior of  solutions of  the stationary Schr\" od\-inger equation with singular potentials, establishing a local decomposition into a homogeneous harmonic polynomial and a lower order term.  Combining a corollary to this result with a  quantitative unique continuation principle for singular potentials we obtain log-H\"older continuity for the density of states outer-measure in one, two, and three dimensions for Schr\" odinger operators  with singular  potentials, results that
hold for the density of states measure when it exists.
\end{abstract}

\maketitle

\section{Introduction}

We study the local behavior of  solutions of  the stationary Schr\" od\-inger equation with singular potentials, establishing a local decomposition into a homogeneous harmonic polynomial and a lower order term. As a corollary, we obtain bounds on the local behavior of  approximate solutions for these equations.  Combining this corollary with a quantitative unique continuation principle for singular potentials \cite{KT},  we obtain log-H\"older continuity for the density of states outer-measure in one, two, and three dimensions for Schr\" odinger operators  with singular  potentials, results that
hold for the density of states measure when it exists.  Our work  extends  results originally proved by Bourgain and Klein \cite{BKl} for bounded potentials.

 Singular potentials introduce technical problems not present for bounded potentials. This can be seen by considering  the Schr\"odinger operator $H= -\Delta + V$.  If $V$ is a bounded potential, i.e., $V\in\mathrm{L}^\infty$, we have  $ \mathcal{D}(H)= \mathcal{D}(-\Delta )\subset\mathrm{H}^2$.  However, if  $V$ is a singular potential, say
  $V\in\mathrm{L}^p$, where   $p\in (d,\infty)$, we only  have  $ \mathcal{D}(H)\subset\mathrm{H}^1$.
   Thus we  have to work with solutions in  $\mathrm{H}^1$, not solutions in  $\mathrm{H}^2$ as in \cite{BKl}.

Let  $\Omega=B(x_0,r)= \{y\in\mathbb{R}^d:|y-x_0|<r\}$,   the ball centered at $x_0\in\mathbb{R}^d$ with radius $r>0$, where $|x| :=(\sum_{j=1}^d|x_j|^2)^{\frac{1}{2}}$ for $x=(x_1,x_2,\ldots,x_d)\in\mathbb{R}^d$. Given a real potential $W\in\mathrm{L}^p(\Omega)$, where  $p\in (d,\infty)$, we consider the  stationary Schr\" odinger equation
\begin{equation}\label{statSch}
-\Delta\phi+W\phi=0 \qtx{a.e. on} \Omega.
\end{equation}
We let   $\mathcal{E}_0(\Omega)$ be the linear space of solutions $\phi\in\mathrm{H}^1(\Omega)$, and
define linear subspaces
\begin{equation}
\mathcal{E}_N(\Omega)=\left\{\phi\in\mathcal{E}_0(\Omega):\limsup_{x\rightarrow x_0}\frac {|\phi(x)|}{|x-x_0|^N}<\infty\right\}\quad\mbox{for}\ N\in\mathbb{N}.
\end{equation}
We have  $\mathcal{E}_1(\Omega)=\{\phi\in\mathcal{E}_0(\Omega):\phi(x_0)=0\}$, and  $\mathcal{E}_N(\Omega)\supset \mathcal{E}_{N+1}(\Omega)$ for all $N\in\mathbb{N}_0=\{0\}\cup\mathbb{N}$. The following theorem is an extension of \cite[Lemma~3.2]{BKl} to singular potentials.  (See \cite{Be,HW} for previous results.)

For  dimensions $d\ge 2$,  let $\mathcal{H}_m^{(d)}$ denote the vector space of homogenous harmonic polynomials on $\mathbb{R}^d$ of degree $m\in\mathbb{N}_0$, and set  $\mathcal{H}_{\leq N}^{(d)}=\bigoplus_{m=0}^N\mathcal{H}_m^{(d)}$. Recall that there exists a constant $\gamma_d>0$ such that (e.g., \cite{ABR})
\begin{equation}\label{conthp}
\dim\mathcal{H}_{\leq N}^{(d)}=\sum_{m=0}^N\dim\mathcal{H}_m^{(d)}\leq\gamma_d N^{d-1}\quad\mbox{for all}\  N\in\mathbb{N}.
\end{equation}
 Constants such as $C_{a,b,\ldots}$  will always  be finite and depending only on the parameters or quantities $a,b,\ldots$; they will be independent of other  parameters or quantities in the equation.  Note that $C_{a,b,\ldots}$ may stand for different constants in different sides of the same inequality.

\begin{theorem}\label{la1Lobe} Let $d=2,3,\ldots$,
$\Omega=B(x_0,3r_0)$ for some $x_0\in\mathbb{R}^d$ and $r_0>0$, .  Fix a real potential $W\in\mathrm{L}^p(\Omega)$, where  $p\in (d,\infty)$, and   set $W_p=\|W\|_{\mathrm{L}^p(\Omega)}$. For all $N\in\mathbb{N}_0$ there exists a linear map $Y_N^{(\Omega)}:\mathcal{E}_N(\Omega)\rightarrow\mathcal{H}_N^{(d)}$ such that for all $\phi\in\mathcal{E}_N(\Omega)$ we have, for all $x\in\overline{B\left(x_0,\frac{r_0}{2}\right)}$, that
\begin{align}\label{l1Lb1}
&|\phi(x)-(Y_N^{(\Omega)}\phi)(x-x_0)|
\\&\quad\quad\leq r_0^{-\frac{d}{2}}(C_{d,p,W_p,r_0})^{N+2}\left(\tfrac{16}{3}\right)^{\frac{(N+1)(N+2)}{2}}((N+1)!)^{d-2}|x-x_0|^{N+1}\|\phi\|_{\mathrm{L}^2(\Omega)}.\nonumber
\end{align}
 As a consequence, for all $N\in\mathbb{N}_0$ we have
\begin{equation}\label{l1Lb2}
\mathcal{E}_{N+1}(\Omega)=\ker Y_N^{(\Omega)}\ \mbox{and}\ \dim\mathcal{E}_{N+1}(\Omega)\geq\dim\mathcal{E}_N(\Omega)-\dim\mathcal{H}_N^{(d)}.
\end{equation}
In particular, if $\mathcal{J}$ is a vector subspace of $\mathcal{E}_0(\Omega)$ we have
\begin{equation}\label{l1Lb3}
\dim\mathcal{J}\cap\mathcal{E}_{N+1}(\Omega)\geq\dim\mathcal{J}-\gamma_dN^{d-1}\ \mbox{for all}\ N\in\mathbb{N},
\end{equation}
where $\gamma_d$ is the constant in \eq{conthp}.
\end{theorem}

As a corollary, we obtain bounds on the local behavior of approximate solutions of the stationary Schr\"odinger equation \eq{statSch} with singular potentials, extending \cite[Theorem~3.1]{BKl}.

\begin{corollary}\label{Lobe}  For  $d=2,3,\ldots$,  let
 $\Omega\subset\mathbb{R}^d$ be an open subset. Let $B(x_0,r_0)\subset\Omega$ for some $x_0\in\mathbb{R}^d$ and $r_0>0$. Fix a real valued function $W\in\mathrm{L}^p(B(x_0,r_0))$ for some $p\in (d,\infty)$. Suppose $\mathcal{F}$ is a linear subspace of $\mathrm{H}^1(\Omega)$ such that for all $\psi\in\mathcal{F}$ we have  $\Delta\psi\in\mathrm{L}^2(B(x_0,r_0))$ and
\begin{equation}\label{tLb1}
\|(-\Delta+W)\psi\|_{\mathrm{L}^{\infty}(B(x_0,r_0))}\leq C_{\mathcal{F}}\|\psi\|_{\mathrm{L}^2(\Omega)}.
\end{equation}
Then there exists  $0<r_1=r_1(d,p,W_p)$, where $W_p=\|W\|_{\mathrm{L}^p(B(x_0,r_0))}$, with the property that for al $N\in\mathbb{N}$ there is a linear subspace $\mathcal{F}_N$ of $\mathcal{F}$, with
\begin{equation}\label{tLb2}
\dim\mathcal{F}_N\geq\dim\mathcal{F}-\gamma_dN^{d-1},
\end{equation}
where $\gamma_d$ is the constant in \eq{conthp}, such that for all $\psi\in\mathcal{F}_N$ we have
\begin{equation}\label{tLb3}
|\psi(x)|\leq(C_{d,p,W_p,r_1}^{N^2}|x-x_0|^{N+1}+C_{\mathcal{F}})\|\psi\|_{\mathrm{L}^2(\Omega)}\quad\mbox{for all}\ x\in B(x_0,r_1).
\end{equation}
\end{corollary}

Equipped with  Corollary~\ref{Lobe} and the quantitative unique continuation  principle  for singular potentials
\cite[Theorem~1.1]{KT},
we establish bounds on the density of states of Schr\"{o}dinger operators $H=-\Delta+V$ on $\mathrm{L}^2(\mathbb{R}^d)$, where now  $\Delta$ is the Laplacian operator, and $V$ is a singular real potential.  In dimensions $d\ge 2$ we will take   $V=V^{(1)}+V^{(2)}$, where  $V^{(1)}\in\mathrm{L}^{\infty}(\mathbb{R}^d)$ and $V^{(2)}\in\mathrm{L}^p(\mathbb{R}^d)$ with  $p\in (d,\infty)$. When applying Corollary~\ref{Lobe}  we use that $\mathrm{L}^{\infty}(\Omega)\subset \mathrm{L}^p(\Omega)$ for $\Omega \subset \R^d$ bounded, in which case
$\mathrm{L}^{\infty}(\Omega) + \mathrm{L}^p(\Omega)= \mathrm{L}^p(\Omega)$.

Given  $\Lambda=\Lambda_L(x)=x+(\tfrac{L}{2},\tfrac{L}{2})^d\subset\mathbb{R}^d$, the open box of side $L>0$ centered at $x\in\mathbb{R}^d$, we let $H_{\Lambda}$ and $\Delta_{\Lambda}$ be the restriction of $H$ and $\Delta$ to $\mathrm{L}^2(\Lambda)$ with Dirichlet boundary condition. The finite volume density of states measure  is given by
\begin{equation}
\eta_{\Lambda}(B):=\frac{1}{|\Lambda|}\tr\{\chi_B(H_{\Lambda})\} \qtx{for Borel sets} B \subset \R^d.
\end{equation}
Recall that for $V$ satisfying appropriate conditions (as in Theorem~\ref{bdsp} below) and  all $E\in \R$  we have
\begin{equation}\label{b}
\eta_{\Lambda}(B)\leq C_{d,V,E}<\infty \qtx{for all Borel sets} B\subset(-\infty,E].
\end{equation}

For periodic and ergodic Schr\"odinger operators, density of states measure $\eta$ can be defined as weak limits of the finite volume density of states measure $\eta_{\Lambda}$ for sequences of boxes $\Lambda\rightarrow\mathbb{R}^d$ in an appropriate sense.
The infinite volume density of states measure cannot be defined for general Schr\"odinger operators, so we follow \cite{BKl} and study  the density of states outer-measure, defined  on Borel subsets $B$ of $\mathbb{R}^d$ by
\begin{align}
\eta^\ast(B):=\limsup_{L\rightarrow\infty}\eta^*_{L}(B),\qtx{where} \eta^*_{L}(B):=\sup_{x\in\mathbb{R}^d}\eta_{\Lambda_L(x)}(B),
\end{align}
 always finite on bounded sets in view of \eqref{b}.

We obtain log-H\"older continuity for the density of states outer-measure of Schr\"od\-inger operators with singular potentials in one, two, and three dimensions,  extending \cite[Theorem~1.1]{BKl}.

\begin{theorem}\label{bdsp}
Let $H=-\Delta+V$ on $\mathrm{L}^2(\mathbb{R}^d)$, where $d=1,2,3$, and $V$ is a real potential such that:
\begin{enumerate}
\item[(i)]if   $d=1$,   $\sup_{x\in\mathbb{R}}\int_{\{|x-y|\leq1\}}|V(y)|dy<\infty$;
\item[(ii)]if  $d=2$,  $V=V^{(1)}+V^{(2)}$, where   $V^{(1)}\in\mathrm{L}^{\infty}(\mathbb{R}^d)$ and $V^{(2)}\in\mathrm{L}^p(\mathbb{R}^d)$ with $p>2$;
\item[(iii)] if $d=3$, $V=V^{(1)}+V^{(2)}$, where    $V^{(1)}\in\mathrm{L}^{\infty}(\mathbb{R}^d)$ and $V^{(2)}\in\mathrm{L}^p(\mathbb{R}^d)$ with $p>6$.
\end{enumerate}
Then, given $E_0\in\mathbb{R}$, for all $E\leq E_0$ and $0<\varepsilon\leq\frac{1}{2}$, we have
\begin{equation}\label{bdsp1}
\eta^{\ast}([E,E+\varepsilon])\leq\frac{C_{d,p,V,E_0}}{\left(\log\frac{1}{\varepsilon}\right)^{\kappa_d}},\quad\mbox{where}\ \kappa_1=1,\kappa_d=\tfrac{(4-d)p-2d}{8p-4d}\ \mbox{for}\ d=2,3.
\end{equation}
\end{theorem}

\section{Local behavior of approximate solutions of the stationary Schr\"{o}dinger equation with singular potentials}

 The fundamental solution to Laplace's equation is given by
\begin{equation}
\Phi(x)=\Phi_d(x):=\left\{
\begin{array}{ll}
(d(d-2)\omega_d)^{-1}|x|^{-d+2}&\mbox{if}\quad d=3,4,\ldots\\-\frac{1}{2\pi}\log|x|&\mbox{if}\quad d=2
\end{array}\right. ,
\end{equation}
 where $\omega_d$ denotes the volume of the unit ball in $\mathbb{R}^d$.

\begin{proof}[Proof of Theorem~\ref{la1Lobe}]
We start as in  \cite[Proof of Lemma~3.2]{BKl}. We take $d=2,3,\ldots$,  and prove the lemma for $\Omega=B(0,3)\subset \R^d$; the general case then follows by translating and dilating. We set $\Omega'=B(0,\tfrac{3}{2})$, and write $\mathcal{E}_n=\mathcal{E}_n(\Omega)$.
Since we only have $\mathcal{E}_0  \subset \mathrm{H}^1(\Omega)$, we must proceed differently from \cite[Proof of Lemma~3.2]{BKl}.
A function  $\phi\in  \mathrm{H}^1(\Omega)$   satisfies an elliptic regularity estimate \cite[Theorem 5.1]{Tr}:
\begin{equation}\label{l1Lbp1}
\|\phi\|_{\mathrm{L}^{\infty}(\Omega')}\leq C_{d,p,W_p}\|\phi\|_{\mathrm{L}^2(\Omega)},
\end{equation}
but we do not have a readily available counterpart for  \cite[Eq.~(3.18)]{BKl}, and thus we must modify the induction.

We fix $\phi\in\mathcal{E}_0$ and consider its Newtonian potential given by
\begin{equation}\label{l1Lbp2}
\psi(x)=-\int_{\Omega'}W(y)\phi(y)\Phi(x-y)dy\quad\mbox{for}\ x\in\mathbb{R}^d.
\end{equation}
Let $q$ be defined by  $\tfrac{1}{p}+\tfrac{1}{q}=1$, so $q<\tfrac{d}{d-1}<\tfrac{d}{d-2}$. Then $\Phi\in\mathrm{L}^{q}(\Omega)$, and it follows from  \eqref{l1Lbp1} that
\begin{equation}\label{l1Lbp3}
|\psi(x)|\leq W_p\|\phi\|_{\mathrm{L}^{\infty}(\Omega')}\|\Phi\|_{\mathrm{L}^q(\Omega)}\leq C_{d,p,W_p}W_p\|\phi\|_{\mathrm{L}^2(\Omega)} \qtx{for all} x\in\Omega'.
\end{equation}
Setting $h=\phi-\psi$, we have $\Delta h=0$ weakly in $\Omega'$, as $\Delta\psi=W\phi$ weakly in $\Omega'$.
It follows that $h$ is a harmonic function in $\Omega'\supset\overline{B(0,1)}$, and, using \cite[Corollary 5.34 and its proof]{ABR}), \eqref{l1Lbp1}, \eqref{l1Lbp3}, \cite[Eqs.~(3.25) and (3.26)]{BKl},  we have that
\begin{equation}\label{l1Lbp4}
h(x)=\sum_{m=0}^{\infty}p_m(x)\sqtx{for all}  x\in B(0,1), \sqtx{where}p_m\in\mathcal{H}_m^{(d)}
\sqtx{for} m=0,1,\ldots,
\end{equation}
with
\begin{equation}\label{l1Lbp7}
|p_m(x)|\leq C_{d,p,W_p}m^{d-2}\|\phi\|_{\mathrm{L}^2(\Omega)}|x|^m\quad\mbox{for all}\ x\in B(0,1), \; m=1,2,\ldots.
\end{equation}
Setting $h_N=\sum_{m=0}^Np_m(x)\in\mathcal{H}_{\leq N}^{(d)}$, it follows that
\begin{equation}\label{l1Lbp8}
|h(x)-h_N(x)|\leq C_{d,p,W_p}\|\phi\|_{\mathrm{L}^2(\Omega)}(N+1)^{d-2}|x|^{N+1}\qtx{for} x\in\overline{B\left(0,\tfrac{1}{2}\right)}.
\end{equation}

 Given $y\in\mathbb{R}^d\backslash\{0\}$, we let $\Phi_y(x)=\Phi(x-y)$.  Since $\Phi_y$ is a harmonic function on $\mathbb{R}^d\backslash\{y\}$,  it is real analytic in $B(0,|y|)$, and we have
(see \cite{ABR})
\begin{equation}\label{l1Lbp10}
\Phi(x-y)=\Phi_y(x)=\sum_{m=0}^{\infty}J_m(x,y)\quad\mbox{for all}\ x\in B(0,|y|),
\end{equation}
where $J_m(\cdot,y)\in\mathcal{H}_m^{(d)}$ for all  $m=0,1,\ldots$, and
the series converges absolutely and uniformly on compact subsets of $B(0,|y|)$. Moreover,  for all $y\in\mathbb{R}^d$ and $m=1,2,\ldots$ we have (see \cite[Corollary 5.34 and its proof]{ABR} and \cite[Eq.~(3.31)]{BKl}) that
\begin{align}\label{l1Lbp11}
|J_m(x,y)|\leq C_dm^{d-2}\left(\frac{4|x|}{3|y|}\right)^m\Phi\left(\frac{y}{4}\right) \qtx{for all} x\in\mathbb{R}^d.
\end{align}
Setting $\Phi_{y,N}(x)=\sum_{m=0}^NJ_m(x,y)\in\mathcal{H}_{\leq N}^{(d)}$, it follows that for $x\in\overline{B\left(0,\frac{1}{2}|y|\right)}$ we have
\begin{equation}\label{l1Lbp12}
|\Phi_y(x)-\Phi_{y,N}(x)|\leq C_d(N+1)^{d-2}\left(\frac{4|x|}{3|y|}\right)^{N+1}\Phi\left(\frac{y}{4}\right).
\end{equation}

We now proceed by induction. We set $\mathcal{E}_{-1}=\mathcal{E}_0$ and $\mathcal{H}_{-1}^{(d)}=\{0\}$. We define $Y_{-1}:\mathcal{E}_{-1}(\Omega)\rightarrow\mathcal{H}_{-1}^{(d)}$ by $Y_{-1}\phi=0$ for all $\phi\in\mathcal{E}_{-1}$.  The theorem holds for $N=-1$ from the elliptic regularity estimate \eqref{l1Lbp1}.

We now let $N\in\mathbb{N}_0$ and suppose that the lemma is valid for $N-1$. If $\phi\in\mathcal{E}_N$, it follows that $\phi\in\mathcal{E}_{N-1}$ with $Y_{N-1}\phi=0$, so by the induction hypothesis
\begin{gather}\label{l1Lbp13}
|\phi(x)|\leq C_N\|\phi(x)\|_{\mathrm{L}^2(\Omega)}|x|^N\quad\mbox{for all}\ \overline{B\left(0,\tfrac{1}{2}\right)},\\
\label{l1Lbp14}
\noindent{\text{where}} \quad C_N=\tilde{C}_{d,p,W_p}^{N+1}\left(\tfrac{16}{3}\right)^{\frac{N(N+1)}{2}}(N!)^{d-2}.
\end{gather}
Using \eqref{l1Lbp11} and \eqref{l1Lbp13}, we define
\begin{equation}\label{l1Lbp15}
\psi_N(x)=-\int_{\Omega'}W(y)\phi(y)\Phi_{y,N}(x)dy\in\mathcal{H}_{\leq N}^{(d)}.
\end{equation}
We fix $x\in\overline{B\left(0,\tfrac{1}{2}\right)}$  and estimate
\begin{equation}\label{l1Lbp16}
|\psi(x)-\psi_N(x)|\leq W_p\left(\int_{\Omega'}(|\phi(y)||\Phi_{y,>N}(x)|)^qdy\right)^{\frac{1}{q}},
\end{equation}
where  $\Phi_{y,>N}(x)=\Phi_y(x)-\Phi_{y,N}(x)$.
From  \eqref{l1Lbp12} and \eqref{l1Lbp13}, with $p>d$, we get
\begin{align}\label{l1Lbp17}
&\left(\int_{\overline{B\left(0,\frac{1}{2}\right)}\backslash B(0,2|x|) }(|\phi(y)||\Phi_{y,>N}(x)|)^qdy\right)^{\frac{1}{q}}
\\&\leq C_dC_N\|\phi\|_{\mathrm{L}^2(\Omega)}(N+1)^{d-2}\left(\tfrac{4}{3}\right)^{N+1}\!\!|x|^{N+1}\!\!\left(\int_{\overline{B\left(0,\frac{1}{2}\right)}\backslash B(0,2|x|)}\left(\tfrac{1}{|y|}\Phi\left(\tfrac{y}{4}\right)\right)^qdy\right)^{\frac{1}{q}}
\nonumber\\&\leq C_{d,p}C_N\|\phi\|_{\mathrm{L}^2(\Omega)}(N+1)^{d-2}\left(\tfrac{4}{3}\right)^{N+1}|x|^{N+1}.\nonumber
\end{align}
If $y\not\in B(0,2|x|)\cup\overline{B\left(0,\tfrac{1}{2}\right)}$ we have $y\geq2|x|$ and $y\geq\tfrac{1}{2}$, and hence, using \eqref{l1Lbp12},
\begin{align}\label{l1Lbp18}
&\left(\int_{\Omega'\backslash\left(B(0,2|x|)\cup\overline{B\left(0,\tfrac{1}{2}\right)}\right)}(|\phi(y)||\Phi_{y,>N}(x)|)^qdy\right)^{\frac{1}{q}}
\\& \qquad \qquad \quad\leq C_d(N+1)^{d-2}\left(\tfrac{8}{3}\right)^{N+1}\Phi\left(\tfrac{1}{8}\right)|x|^{N+1}\left(\int_{\Omega'}|\phi(y)|^q\right)^{\frac{1}{q}}
\nonumber\\& \qquad \qquad \quad  \leq C_d(N+1)^{d-2}\left(\tfrac{8}{3}\right)^{N+1}|x|^{N+1}\|\phi\|_{\mathrm{L}^2(\Omega)}.\nonumber
\end{align}
Using \eqref{l1Lbp11} and \eqref{l1Lbp13}, we get
\begin{align}\label{l1Lbp19}
&\left(\int_{B(0,2|x|)\cap\overline{B\left(0,\tfrac{1}{2}\right)}}(|\phi(y)||\Phi_{y,>N}(x)|)^qdy\right)^{\frac{1}{q}}
\\&\leq C_N\|\phi\|_{\mathrm{L}^2(\Omega)}\left(\int_{B(0,2|x|)\cap\overline{B\left(0,\tfrac{1}{2}\right)}}(|y|^N|\Phi_{y,>N}(x)|)^qdy\right)^{\frac{1}{q}}
\nonumber\\&\leq C_N\|\phi\|_{\mathrm{L}^2(\Omega)}\left(\int_{B(0,2|x|)\cap\overline{B\left(0,\tfrac{1}{2}\right)}}(|y|^N|\Phi(x-y)|)^qdy\right)^{\frac{1}{q}}
\nonumber\\&\; \;+C_dC_N\|\phi\|_{\mathrm{L}^2(\Omega)}\!\!\sum_{m=0}^Nm^{d-2}\left(\tfrac{4}{3}|x|\right)^m\!\!\left(\int_{B(0,2|x|)\cap\overline{B\left(0,\tfrac{1}{2}\right)}}\left(|y|^{N-m}\!\left|\Phi\left(\tfrac{y}{4}\right)\right|\right)^qdy\right)^{\frac{1}{q}}
\nonumber\\&\leq C_dC_N\|\phi\|_{\mathrm{L}^2(\Omega)}\left(2^N+N^{d-2}\left(\tfrac{4}{3}\right)^{N+1}\right)|x|^{N+1},\nonumber
\end{align}
where we used $\frac{3|x|}{|x-y|}\geq1 $ for $y\in B(0,2|x|)$. (Note that we get $|x|^{N+2-\frac{d}{p}}$ if $d\geq3$ and $|x|^{\left(N+2-\frac{d}{p}\right)-}$ if\ $d=2$.) Also using \eqref{l1Lbp11}, we get
\begin{align}\label{l1Lbp20}
&\left(\int_{\Omega'\backslash\overline{B\left(0,\tfrac{1}{2}\right)}}(|\phi(y)||\Phi_{y,>N}(x)|)^qdy\right)^{\frac{1}{q}}
\\& \qquad \leq \left(\int_{\Omega'\backslash\overline{B\left(0,\tfrac{1}{2}\right)}}(|\phi(y)||\Phi(x-y)|)^qdy\right)^{\frac{1}{q}}
\nonumber\\&\qquad \qquad+C_d\sum_{m=0}^Nm^{d-2}\left(\tfrac{4}{3}|x|\right)^m\left(\int_{\Omega'\backslash\overline{B\left(0,\frac{1}{2}\right)}}\left(|\phi(y)||y|^{-m}\left|\Phi\left(\tfrac{y}{4}\right)\right|\right)^qdy\right)^{\frac{1}{q}}
\nonumber\\&\qquad \leq C_{d,p,W_p}\|\phi\|_{\mathrm{L}^2(\Omega)}\left(1+N^{d-2}\left(\tfrac{4}{3}\right)^{N+1}\right),\nonumber
\end{align}
where we used $|x|\leq\tfrac{1}{2}$. Since $|x|>\tfrac{1}{4}$ if $y\in B(0,2|x|)\backslash\overline{B\left(0,\tfrac{1}{2}\right)}$, we obtain
\begin{align}\label{l1Lbp21}
&\left(\int_{(\Omega'\cap B(0,2|x|)) \backslash\overline{B\left(0,\frac{1}{2}\right)}}(|\phi(y)||\Phi_{y,>N}(x)|)^qdy\right)^{\frac{1}{q}}
\\&\qquad \qquad \qquad \leq C_{d,p,W_p}\|\phi\|_{\mathrm{L}^2(\Omega)}\left(4^{N+1}+N^{d-2}\left(\tfrac{16}{3}\right)^{N+1}\right)|x|^{N+1}.\nonumber
\end{align}
Combining \eqref{l1Lbp16}, \eqref{l1Lbp17}, \eqref{l1Lbp18}, \eqref{l1Lbp19} and \eqref{l1Lbp21},   we have ($C_N\geq1$)  \begin{equation}\label{l1Lbp22}
|\psi(x)-\psi_N(x)|\leq C_{d,p,W_p}C_NW_p(N+1)^{d-2}|x|^{N+1}\|\phi\|_{\mathrm{L}^2(\Omega)},
\end{equation}
for all $x\in\overline{B\left(0,\frac{1}{2}\right)}$.

Now let  $Y_N\phi=h_N+\psi_N\in\mathcal{H}_N^{(d)}$. It  follows from \eqref{l1Lbp8}, \eqref{l1Lbp22} and \eqref{l1Lbp14}, choosing the constant $\tilde{C}_{d,p,W_p}$ in \eqref{l1Lbp14} large enough, that  for all $x\in\overline{B\left(0,\frac{1}{2}\right)}$ we have
\begin{align}\label{l1Lbp23}
&|\phi(x)-(Y_N\phi)(x)|
\leq|h(x)-h_N(x)|+|\psi(x)-\psi_N(x)|
\nonumber\\&\leq(C_{d,p,W_p}+C_{d,p}W_p C_N)(N+1)^{d-2}\left(\tfrac{16}{3}\right)^{N+1}|x|^{N+1}\|\phi\|_{\mathrm{L}^2(\Omega)}
\nonumber\\&\leq \tilde{C}_{d,p,W_p}C_N(N+1)^{d-2}\left(\tfrac{16}{3}\right)^{N+1}|x|^{N+1}\|\phi\|_{\mathrm{L}^2(\Omega)}
\nonumber\\&\leq \tilde{C}_{d,p,W_p}\left(\tilde{C}_{d,p,W_p}^{N+1}\left(\tfrac{16}{3}\right)^{\frac{N(N+1)}{2}}\!\!(N!)^{d-2}\right)(N+1)^{d-2}\left(\tfrac{16}{3}\right)^{N+1}\!|x|^{N+1}\|\phi\|_{\mathrm{L}^2(\Omega)}
\nonumber\\&\leq \tilde{C}_{d,p,W_p}^{N+2}\left(\tfrac{16}{3}\right)^{\frac{(N+1)(N+2)}{2}}((N+1)!)^{d-2}|x|^{N+1}\|\phi\|_{\mathrm{L}^2(\Omega)}.\nonumber
\end{align}
 This completes the induction.

Since  \eqref{l1Lb2} is a consequence of \eqref{l1Lb1}, and \eqref{l1Lb3} follows from \eqref{l1Lb2}. the theorem is proven.
\end{proof}

Corollary \ref{Lobe} is an immediate consequence from the following corollary.

\begin{corollary}  For $d=2,3,\ldots$,
let $\Omega\subset\mathbb{R}^d$ be an open subset.   Let $B(x_0,r_1)\subset\Omega$ for some $x_0\in\mathbb{R}^d$ and $r_1>0$. Fix a real valued function $W\in\mathrm{L}^p(B(x_0,r_1))$ for some $p\in (d,\infty)$. Suppose $\mathcal{F}$ is a linear subspace of $\mathrm{H}^1(\Omega)$ such that for all $\psi\in\mathcal{F}$ we have  $\Delta\psi\in\mathrm{L}^2(B(x_0,r_1))$ and
\begin{equation}\label{l2Lb1}
\|(-\Delta+W)\psi\|_{\mathrm{L}^{\infty}(B(x_0,r_1))}\leq C_{\mathcal{F}}\|\psi\|_{\mathrm{L}^2(\Omega)}.
\end{equation}
Then there exists $0<r_2=r_2(d,p,W_p)<r_1$, where $W_p=\|W\|_{\mathrm{L}^p(B(x_0,r_1))}$, with the property that for all $r\in(0,r_2]$ there is a linear map $Z_r:\mathcal{F}\rightarrow\mathcal{E}_0(B(x_0,r))$ such that
\begin{equation}\label{l2Lb2}
\|\psi-Z_r\psi\|_{\mathrm{L}^{\infty}(B(x_0,r))}\leq C_{d,r}C_{\mathcal{F}}\|\psi\|_{\mathrm{L}^2(\Omega)}, \quad\mbox{where}\ \lim_{r\rightarrow0}C_{d,r}=0.
\end{equation}
As a consequence, for all $N\in\mathbb{N}$ there is a vector subspace $\mathcal{F}_N$ of $\mathcal{F}$, with
\begin{equation}\label{l2Lb3}
\dim\mathcal{F}_N\geq\dim\mathcal{F}-\gamma_dN^{d-1},
\end{equation}
such that for all $\psi\in\mathcal{F}_N$ we have
\begin{equation}\label{l2Lb4}
|\psi(x)|\leq(C_{d,p,W_p,r_1}^{N^2}|x-x_0|^{N+1}+C_{\mathcal{F}})\|\psi\|_{\mathrm{L}^2(\Omega)}\quad\mbox{for all}\ x\in \overline{B(x_0,\tfrac{r_2}{6})}.
\end{equation}
\end{corollary}

\begin{proof}
We proceed as in  \cite[Lemma~3.3]{BKl}. It suffices to consider $x_0=0$. We set $B_r=B(0,r)$. Given $0<r<r_1$ and $\psi\in \mathrm{H}^1(\Omega)$ with  $\Delta\psi\in\mathrm{L}^2(B_r)$, we define $Z_r\psi\in\mathcal{E}_0(B_r)$ as the unique solution $\phi\in\mathrm{H}^1(B_r)$ to the Dirichlet problem on $B_r$ given by
\begin{equation}\label{l2Lbp1}
\left\{
\begin{array}{ll}
-\Delta\phi+W\phi=0&\mbox{on}\ B_r,\\
\phi=\psi&\mbox{on}\ \partial B_r.
\end{array}
\right.
\end{equation}
This map is well defined in view of \cite[Theorem 3.2]{Tr}.  (Since $W\in\mathrm{L}^p(B_r)$ for some $p\in (d,\infty)$, $|W|$ is compactly bounded on $H_0^1(B_r)$ by \cite[Lemma~1.4]{Tr}.  Moreover, for  $\psi\in\mathrm{H}^1(\Omega)$ with  $\Delta\psi\in\mathrm{L}^2(B_r)$ we have $\|\nabla\psi\|^2_{\mathrm{L}^2(B_r)}+\int_{B_r}|W|\abs{\psi}^2dx<\infty$, as in \cite[Eq. (2.21) and (2.46)]{KT}. Therefore \cite[Theorem 3.2]{Tr} can be applied.) It is clearly a linear map.

To prove \eqref{l2Lb2}, we use the Green's function $G_r(x,y)$ for the ball $B_r$ (see \cite[Section 2.5]{GiT}),
\begin{equation}\label{l2Lbp2}
G_r(x,y)=\left\{
\begin{array}{ll}
\Phi(|x-y|)-\Phi(\frac{|y|}{r}|x-\frac{r^2}{|y|^2}y|)&\mbox{if}\ y\neq0,\\
\Phi(|x|)-\Phi(r)&\mbox{if}\ y=0.
\end{array}
\right.
\end{equation}
Let $\psi\in\mathcal{F}$.  Using Green's representation formula \cite[Eq. (2.21)]{GiT} for $\psi$ and $Z_r\psi$, for all $x\in B_r$ we have
\begin{align}\label{l2Lbp3}
\psi(x)=&-\int_{\partial B_r}\psi(\zeta)\partial_{\nu}G_r(x,\zeta)dS(\zeta)-\int_{B_r}W(y)\psi(y)G_r(x,y)dy\\&\quad+\int_{B_r}\pa{(-\Delta+W)\psi}\!(y)\, G_r(x,y)dy,
\\ (Z_r\psi)(x)=&-\int_{\partial B_r}\psi(\zeta)\partial_{\nu}G_r(x,\zeta)dS(\zeta)-\int_{B_r}W(y)(Z_r\psi)(y)G_r(x,y)dy,\nonumber
\end{align}
where $dS$ denotes the surface measure and $\partial_{\nu}$ is the normal derivative.  For all $x\in B_r$ an explicit calculation gives
\begin{align}\label{l2Lbp4}
\|G_r(x,\cdot)\|_{\mathrm{L}^1(B_r)}&\leq C'_dr^{\frac{d(\alpha_d-1)}{\alpha_d}}\|G_r(x,\cdot)\|_{\mathrm{L}^{\alpha_d}(B_r)}\leq C_dr^{\frac{d(\alpha_d-1)}{\alpha_d}},\\
\label{l2Lbp5}
\|G_r(x,\cdot)\|_{\mathrm{L}^q(B_r)}&\leq C'_dr^{\frac{d(\alpha_d-q)}{\alpha_dq}}\|G_r(x,\cdot)\|_{\mathrm{L}^{\alpha_d}(B_r)}\leq C_dr^{\frac{d(\alpha_d-q)}{\alpha_dq}},
\end{align}
where  $\alpha_2=2$ and $\alpha_d=\frac{d-1}{d-2}$ for $d\geq3$, and  $\tfrac{1}{p}+\tfrac{1}{q}=1$ ($q<\frac{d}{d-1}\leq\alpha_d$ as $p>d$). We conclude that
\begin{align}\label{l2Lbp6}
&\|\psi-Z_r\psi\|_{\mathrm{L}^{\infty}(B_r)}  \\  \notag
& \qquad \leq C_dr^{\frac{d(\alpha_d-q)}{\alpha_dq}}W_p\|\psi-Z_r\psi\|_{\mathrm{L}^{\infty}(B_r)}+C_dr^{\frac{d(\alpha_d-1)}{\alpha_d}}\|(-\Delta+W)\psi\|_{\mathrm{L}^{\infty}(B_r)}.
\end{align}
Taking $r_2\in(0,r_1)$ such that $C_dr^{\frac{d(\alpha_d-q)}{\alpha_dq}}(1+W_p)\leq\tfrac{1}{2}$, and using \eqref{l2Lb1}, we get \eqref{l2Lb2}.

Letting  $\mathcal{J}=\Ran Z_{r_2}$, and setting $\mathcal{J}_N=\mathcal{J}\cap\mathcal{E}_{N+1}(B_{r_2})$, $\mathcal{F}_N=Z_{r_2}^{-1}(\mathcal{J}_N)$,  the estimate  \eqref{l2Lb4} follows  using the  argument in \cite[Lemma~3.3]{BKl}.
\end{proof}

\section{Bounds on the Density of States of Schr\"{o}dinger Operators with Singular Potentials}

\subsection{One-dimensional Schr\"{o}dinger operators with singular potentials}

The case $d=1$ of Theorem \ref{bdsp} is an immediate consequence of the following theorem.
\begin{theorem}\label{thm1d}
Let $H=-\Delta+V$ on\ $\mathrm{L}^2(\mathbb{R})$, where $V$ is a real potential such that\begin{equation}\label{t1d1}
\sup_{x\in\mathbb{R}}\int_{\{|x-y|\leq1\}}|V(y)|dy<\infty.
\end{equation}
Given $E_0\in\mathbb{R}$, there exists $L_{V,E_0}$ such that for all $0<\varepsilon\leq\frac{1}{2}$, open intervals $\Lambda=\Lambda_L$ with $L\geq L_{V,E_0}\log\frac{1}{\varepsilon}$, and $E\leq E_0$, we have
\begin{equation}\label{t1d2}
\eta_{\Lambda}([E,E+\varepsilon])\leq\frac{C_{V,E_0}}{\log\frac{1}{\varepsilon}}.
\end{equation}
\end{theorem}

\begin{proof}
Proceeding as in  \cite[Theorem~2.3]{BKl},
let $\Lambda=\Lambda_L=(a_0,a_0+L)$, $E\in\mathbb{R}$, $\varepsilon\in(0,\frac{1}{2}]$ and
\begin{equation}\label{t1dp1}
K=\sup_{x\in\mathbb{R}}\int_{\{|x-y|\leq1\}}|V(y)|dy<\infty.
\end{equation}
Setting  $P=\chi_{[E,E+\varepsilon]}(H_{\Lambda})$, we have $\Ran P\subset\mathcal{D}(H_{\Lambda})\subset\mathrm{C}^1(\Lambda)$, and
\begin{equation}\label{t1dp2}
\|(H_{\Lambda}-E)\psi\|_2\leq\varepsilon\|\psi\|_2\quad\mbox{for all}\ \psi\in\Ran P.
\end{equation}

Given $0<R<L$, set $a_j=a_0+jR$ for $j=1,2,\ldots,\left\lceil\tfrac{L}{R}\right\rceil-1$, and consider  the vector space
\begin{equation}\label{t1dp3}
\mathcal{F}_R:=\left\{\psi\in\Ran P:\psi(a_j)=\psi'(a_j)=0\quad\mbox{for}\ j=1,2,\ldots,\left\lceil\tfrac{L}{R}\right\rceil-1\right\}.
\end{equation}
Given\ $\psi\in\mathcal{F}_R$, set $\Psi=\left(
\begin{array}{c}
\psi\\\psi'
\end{array}
\right)$.
We have
\begin{equation}\label{t1dp5}
\Psi'=\left(
\begin{array}{c}
\psi'\\\psi''
\end{array}
\right)=\left(
\begin{array}{c}
\psi'\\V\psi-H\psi
\end{array}
\right)=\left(
\begin{array}{cc}
0&1\\V-E&0
\end{array}\right)\Psi+
\left(\begin{array}{c}
0\\-\zeta
\end{array}\right)
\end{equation}
where $\zeta=(H-E)\psi$. We have $\|\zeta\|_2\leq\varepsilon\|\psi\|_2$ from \eqref{t1dp2}. For $j=1,2,\ldots,\left\lceil\tfrac{L}{R}\right\rceil-1$ and $x\in(a_j-R,a_j+R)\cap\Lambda$, we have
\begin{equation}\label{t1dp6}
\Psi(x)=\int_{a_j}^x\left(\begin{array}{cc}
0&1\\(V(y)-E)&0
\end{array}\right)\Psi(y)dy+\int_{a_j}^x\left(\begin{array}{c}
0\\-\zeta(y)
\end{array}\right)dy
\end{equation}
since $\psi(a_j)=\psi'(a_j)=0$, and hence
\begin{equation}\label{t1dp7}
|\Psi(x)|\leq\left|\int_{a_j}^x(1+|E|+|V(y)|)|\Psi(y)|)dy+\int_{a_j}^x|\zeta(y)|dy\right|.
\end{equation}
By Gronwall's inequality (see \cite{Ho}), we have
\begin{equation}\label{t1dp8}
|\Psi(x)|\leq\left|\int_{a_j}^x\exp\left(\left|\int_y^x(1+|E|+|V(z)|)dz\right|\right)|\zeta(y)|dy\right|.
\end{equation}
We have
\begin{align}\label{t1dp9}
\left|\int_y^x(1+|E|+|V(z)|)dz\right| &\leq(1+|E|)|x-y|+\left|\int_y^x|V(z)dz|\right|
\\   &  \leq(1+|E|)R+\left\lceil\tfrac{R}{2}\right\rceil K
\leq C\max\{R,1\}, \nonumber
\end{align}
where $C=1+|E|+K$. Therefore
\begin{equation}\label{t1dp10}
|\psi(x)|\leq|\Psi(x)|\leq e^{C\max\{R,1\}}\sqrt{|x-a_j|}\|\zeta\|_2\leq e^{C\max\{R,1\}}\sqrt{R}\varepsilon\|\psi\|_2.
\end{equation}
Since $\Lambda$ is the union of these intervals, we conclude that
\begin{equation}\label{t1dp11}
\|\psi\|_{\infty}\leq e^{C\max\{R,1\}}\sqrt{R}\varepsilon\|\psi\|_2\quad\mbox{for all}\ \psi\in\mathcal{F}_R.
\end{equation}

We now assume that
\begin{equation}\label{t1dp12}
\rho:=\eta_{\Lambda_L}([E,E+\varepsilon])=\tfrac{1}{L}\tr P>\tfrac{4}{L},
\end{equation}
since otherwise there is nothing to prove for large $L$. Taking $R=\tfrac{4}{\rho}$, it follows from\ \eqref{t1dp12}\ that
\begin{equation}\label{t1dp13}
\dim\mathcal{F}_R\geq\rho L-2\left(\left\lceil\tfrac{L}{R}\right\rceil-1\right)\geq\rho L-2\tfrac{L}{R}=\tfrac{1}{2}\rho L>2.
\end{equation}
Applying \cite[Lemma~2.1]{BKl}, we obtain $\psi_0\in\mathcal{F}_R$, $\psi_0\neq0$, such that
\begin{equation}\label{t1dp14}
\|\psi_0\|_{\infty}\geq\sqrt{\frac{\dim\mathcal{F}_R}{L}}\|\psi_0\|_2\geq\sqrt{\tfrac{1}{2}\rho}\|\psi_0\|_2.
\end{equation}
It follows from \eqref{t1dp11} and \eqref{t1dp14} that
\begin{equation}\label{t1dp15}
\sqrt{\tfrac{1}{2}\rho}\leq e^{C\max\{R,1\}}\sqrt{R}\varepsilon=e^{C(\max\{\tfrac{4}{\rho},1\})}\sqrt{\tfrac{4}{\rho}}\varepsilon.
\end{equation}
If $\rho\leq4$, we have  $\tfrac{4}{\rho}\geq1$,  and we get
\begin{equation}\label{t1dp16}
\rho\leq\frac{8C}{\log\frac{1}{\varepsilon}}.
\end{equation}
If $\rho>4$, we have  $\frac{4}{\rho}<1$, and we get
\begin{equation}\label{t1dp17}
\rho\leq 2\sqrt{2}e^C\varepsilon\leq\frac{2\sqrt{2}e^C}{\log\frac{1}{\varepsilon}}.
\end{equation}
Since we have \eq{t1dp12}, we conclude that there exists $C_{K,E}$ such that
\begin{equation}\label{t1dp18}
\rho\leq\frac{C_{K,E}}{\log\frac{1}{\varepsilon}} \qtx{if}  L>\frac{4}{\rho}\geq\frac{4\log\frac{1}{\varepsilon}}{C_{K,E}}.
\end{equation}

Since $H_{\Lambda}$ is semibounded (see \cite{Si}), there exists $\theta_{V}$ such that $\sigma(H_{\Lambda})\subset[\theta_{V},\infty)$. Thus we have $\eta_{\Lambda}([E,E+\varepsilon])=0$ unless $E\geq\theta_{V}-\frac{1}{2}$. Thus, given $E_0\in\mathbb{R}$, there exists $L_{V,E_0}$ such that, for all $0<\varepsilon\leq\frac{1}{2}$, open intervals $\Lambda=\Lambda_L$ with $L\geq L_{V,E_0}\log\frac{1}{\varepsilon}$, and $E\leq E_0$, we have \eqref{t1d2}.
\end{proof}
\subsection{Two and three dimensional Schr\"{o}dinger operators with singular potentials}

We start by recalling a quantitative unique continuation principle for Schr\"{o}dinger operators with singular potentials \cite{KT}, an extension of the bounded potentials results  of \cite{BK,GKloc,BKl}.  We state   only what we  use in the proof of  Theorem \ref{bdsp}.
Given subsets $A$ and $B$ of $\mathbb{R}^d$, and a function $\varphi$ on set $B$, we set  $\varphi_A:=\varphi\chi_{A\cap B}$. We let $\varphi_{x,\delta}:=\varphi_{B(x,\delta)}$.

\begin{theorem}[{\cite[Theorem~1.1]{KT}}]\label{qucp3d}  Let $d=2,3,\ldots$.
Let $\Omega$ be an open subset of $\mathbb{R}^d$, and consider a real measurable function $V=V^{(1)}+V^{(2)}$ on $\Omega$ with $\|V^{(1)}\|_{\infty}\leq K_1$ and    $\|V^{(2)}\|_p\leq K_2$,  with either  $p\geq d$ if $d\geq3$ or
   $p>2$ if $d=2$.   Set  $K=K_1+K_2$.   Let $\psi\in\mathrm{L}^2(\Omega)$ be real valued with $\Delta \psi\in\mathrm{L}_{loc}^2(\Omega)$, and suppose
\begin{equation}\label{q3d1}
\zeta=-\Delta\psi+V\psi\in\mathrm{L}^2(\Omega).
\end{equation}
Let $\Theta\subset\Omega$ be a bounded measurable set where\ $\|\psi_{\Theta}\|_{2}>0$, and set
\begin{equation}\label{q3d2}
Q(x,\Theta):=\sup_{y\in\Theta}|y-x|\quad\mbox{for}\quad x\in\Omega.
\end{equation}
Consider $x_0\in\Omega\backslash\overline{\Theta}$ such that
\begin{equation}\label{q3d3}
Q=Q(x_0,\Theta)\geq1\quad \mbox{and}\quad B(x_0,6Q+2)\subset\Omega,
\end{equation}
and take
\begin{equation}\label{q3d4}
0<\delta\leq\min\{\dist(x_0,\Theta),\tfrac{1}{2}\}.
\end{equation}
There is a constant $m_d>0$, depending only on $d$, such that\begin{equation}\label{q3d5}
\left(\tfrac{\delta}{Q}\right)^{m_d(1+K^{\frac{2p}{3p-2d}})(Q^{\frac{4p-2d}{3p-2d}}+\log\frac{\|\psi_{\Omega}\|_{2}}{\|\psi_{\Theta}\|_{2}})}\|\psi_{\Theta}\|_{2}^2\leq\|\psi_{x_0,\delta}\|_{2}^2+\delta^2\|\zeta_{\Omega}\|_{2}^2.
\end{equation}
 \end{theorem}

As noted in \cite[Corollary A.2]{GKloc}, when we apply Theorem \ref{qucp3d} to approximate eigenfunction of Schr\"{o}dinger operators defined on a box $\Lambda$ with Dirichlet or periodic boundary condition, it can be extended to sites near the boundary of $\Lambda$ as in the following corollary.

\begin{corollary}\label{qucpcor} Let $d=2,3,\ldots$.
Consider the Schr\"{o}dinger operator $H_{\Lambda}:=-\Delta_{\Lambda}+V$ on $\mathrm{L}^2(\Lambda)$, where $\Lambda=\Lambda_L(x_0)$ is the open box of side $L>0$ centered at $x_0\in\mathbb{R}^d$. $\Delta_{\Lambda}$\ is the Laplacian with either Dirichlet or periodic boundary condition on $\Lambda$, and $V=V^{(1)}+V^{(2)}$ is a real potential on $\Lambda$ with $\|V^{(1)}\|_{\infty}\leq K_1<\infty$ and $\|V^{(2)}\|_p\leq K_2<\infty$, with either  $p\geq d$ if $d\geq3$ or
   $p>2$ if $d=2$. Let $\psi\in\mathcal{D}(H_{\Lambda})$ with $\Delta \psi\in\mathrm{L}^2(\Lambda)$ and fix a bounded measurable set $\Theta\subset\Lambda$ where $\|\psi_{\Theta}\|_{2}>0$. Set
$Q(x,\Theta):=\sup_{y\in\Theta}|y-x|$ for $x\in\Lambda$, and consider $x_0\in\Omega\backslash\overline{\Theta}$ such that $Q=Q(x_0,\Theta)\geq1$. Then, given $0<\delta\leq\min\{\dist(x_0,\Theta),\frac{1}{2}\}$, such that $B(x_0,\delta)\subset\Lambda$, we have
\begin{equation}
\left(\tfrac{\delta}{Q}\right)^{m_d(1+K^{\frac{2p}{3p-2d}})(Q^{\frac{4p-2d}{3p-2d}}+\log\frac{\|\psi\|_{2}}{\|\psi_{\Theta}\|_{2}})}\|\psi_{\Theta}\|_{2}^2\leq\|\psi_{x_0,\delta}\|_{2}^2+\delta^2\|H_{\Lambda}\psi\|_2^2,
\end{equation}
where $K=K_1+K_2$ and $m_d>0$ is a constant depending only on $d$.
\end{corollary}
This corollary is proved exactly as \cite[Corollary A.2]{GKloc}. (Note that using the notation in the proof of \cite[Corollary A.2]{GKloc}, we have $\|\widehat{V^{(1)}}_{\Lambda_{L'}}\|_{\infty}=\|V^{(1)}_{\Lambda_L}\|_{\infty}$ and $\|\widehat{V^{(2)}}_{\Lambda_{L'}}\|_p\leq(2n+1)^d\|V^{(2)}_{\Lambda_L}\|_p$ if $L'=(2n+1))L$ for some $n\in\mathbb{N}$.)

The case $d=2,3$ of Theorem \ref{bdsp} is an immediate consequence of the following theorem.
\begin{theorem}
Let $H=-\Delta+V$ on $\mathrm{L}^2(\mathbb{R}^d)$, where $d=2,3$ and $V=V^{(1)}+V^{(2)}$ is a real potental with    $V^{(1)}\in\mathrm{L}^{\infty}(\mathbb{R}^d)$ and $V^{(2)}\in\mathrm{L}^p(\mathbb{R}^d)$ with  $p>\frac{2d}{4-d}$. Set $V_{\infty}=\|V\|_{\infty}$ and $V_p=\|V\|_p$. Given $E_0\in\mathbb{R}$, there exists $L_{d,p,V^{(1)}_{\infty},V^{(2)}_p,E_0}$ such that for all $0<\varepsilon\leq\frac{1}{2}$, open boxes $\Lambda=\Lambda_L$ with $L\geq L_{d,p,V_p,E_0}\left(\log\frac{1}{\varepsilon}\right)^{\frac{3p-2d}{8p-4d}}$, and\ $E\leq E_0$, we have
\begin{equation}\label{tmd1}
\eta_{\Lambda}([E,E+\varepsilon])\leq\frac{C_{d,p,V^{(1)}_{\infty},V^{(2)}_p,E_0}}{\left(\log\frac{1}{\varepsilon}\right)^{\frac{(4-d)p-2d}{8p-4d}}}.
\end{equation}
\end{theorem}

\begin{proof}
We fix $\varepsilon\in(0,\frac{1}{2}]$, let $L\geq L_0(\varepsilon)$, where $L_0(\varepsilon)>0$ will be specified later, and take a box $\Lambda=\Lambda_L$. There exists $\theta=\theta(d,p,V^{(1)}_{\infty},V^{(2)}_p)\geq0$  such that  (see \cite[Eq. (2.21) and (2.46)]{KT})
\begin{equation}\label{tmdp1}
\left|\int_{\mathbb{R}^d}\abs{V}\abs{f}^2dx\right|\leq\theta\|f\|_2^2+\tfrac{1}{2}\|\nabla f\|_2^2 \qtx{for all}f\in\mathcal{D}(\nabla).
\end{equation}
It follows that  $\sigma(H_{\Lambda})\subset[-\theta,\infty)$, and hence it suffices to consider $E_0\geq-\theta-1$ and $E\in[-\theta-1,E_0]$. We set $P=\chi_{[E,E+\varepsilon]}(H_{\Lambda})$; note that $\Ran P\subset\mathcal{D}(H_{\Lambda})\subset\mathrm{H}^1(\Lambda)$ and
\begin{equation}\label{tmdp2}
\|(H_{\Lambda}-E)\psi\|_2\leq\varepsilon\|\psi\|_2\quad\mbox{for all}\ \psi\in\Ran P.
\end{equation}
Recalling  that for $t>0$ we have
\begin{align}\label{tmdp4}
\|e^{-t(H_{\Lambda}+\theta)}\|_{\mathrm{L}^2(\Lambda)\rightarrow\mathrm{L}^{\infty}(\Lambda)}&\leq\|e^{\frac{1}{2}t\Delta_{\Lambda}}\|_{\mathrm{L}^2(\Lambda)\rightarrow\mathrm{L}^{\infty}(\Lambda)}
\nonumber\\&\leq\|e^{\frac{1}{2}t\Delta}\|_{\mathrm{L}^2(\mathbb{R}^d)\rightarrow\mathrm{L}^{\infty}(\mathbb{R}^d)}<\infty,
\end{align}
for $\psi\in\Ran P$ we get
\begin{align}\label{tmdp3}
\|\psi\|_{\infty}&=\|e^{-(H_{\Lambda}+\theta)}e^{(H_{\Lambda}+\theta)}\psi\|_{\infty} \\&
\leq\|e^{-(H_{\Lambda}+\theta)}\|_{\mathrm{L}^2(\Lambda)\rightarrow\mathrm{L}^{\infty}(\Lambda)}\|e^{(H_{\Lambda}+\theta)}\psi\|_2
\leq C_d e^{E_0+\theta+1}\|\psi\|_2. \nonumber
\end{align}
Since $P(H_{\Lambda}-E)\psi=(H_{\Lambda}-E)P\psi=(H_{\Lambda}-E)\psi$\ for\ $\psi\in\Ran P$, we conclude that
\begin{equation}\label{tmdp5}
\|(H_{\Lambda}-E)\psi\|_{\infty}\leq\varepsilon C_{d,p,V^{(1)}_{\infty},V^{(2)}_p,E_0}\|\psi\|_2\quad\mbox{for all}\ \psi\in\Ran P.
\end{equation}
Since $V\in\mathrm{L}^\infty(\mathbb{R}^d)+\mathrm{L}^p(\mathbb{R}^d)$ with $p>2$, we have $V\in\mathrm{L}_{loc}^2(\mathbb{R}^d)$. Therefore $V\psi\in\mathrm{L}^2(\Lambda)$ as $\psi$ is bounded. Thus we have $\Delta\psi=-H_{\Lambda}\psi+V\psi \in \mathrm{L}^2(\Lambda)$.

Let
\begin{equation}\label{tmdp6}
\rho:=\eta_{\Lambda_L}([E,E+\varepsilon])=\tfrac{1}{L^d}\tr P.
\end{equation}
We have the uniform upper bound (e.g., \cite[Eq. (A.6)]{GK1})
\begin{equation}\label{tmdp7}
\rho\leq\rho_{ub}:=C_{d,p,V^{(1)}_{\infty},V^{(2)}_p,E_0}; \qtx{without loss of generality} \rho_{ub}\geq1.
\end{equation}

Let $\gamma_d$ be the constant in Theorem \ref{Lobe}; we assume $2^d\gamma_d\geq1$ without loss of generality.  We take
\begin{equation}\label{tmdp8}
L^d>2^{3d+1}\gamma_d\frac{\rho_{ub}}{\rho};
\end{equation}
 otherwise there is nothing to prove for $L$  large. Let  $R$ satisfy
\begin{equation}\label{tmdp9}
2^{d+1}\gamma_d\frac{\rho_{ub}}{\rho}\leq R^d<\left(\frac{L}{4}\right)^d;
\end{equation}
we have
\begin{equation}\label{tmdp10}
2\leq\rho R^d\ \mbox{and}\ 2\leq R^d.
\end{equation}
Using \eqref{tmdp7} and \eqref{tmdp9}, we have
\begin{equation}\label{tmdp11}
N:=\left\lfloor\left(\frac{\rho}{2^{d+1}\gamma_d}\right)^{\frac{1}{d-1}}R^{\frac{d}{d-1}}\right\rfloor\geq\left\lfloor\rho_{ub}^{\frac{1}{d-1}}\right\rfloor\geq1.
\end{equation}

We now choose  $\mathcal{G}\subset\Lambda$ such that
\begin{equation}\label{tmdp12}
\overline{\Lambda}=\bigcup_{y\in\mathcal{G}}\overline{\Lambda_R}(y)\quad\mbox{and}\quad\sharp\mathcal{G}=\left(\left\lceil\tfrac{L}{R}\right\rceil\right)^d\in\left[\left(\tfrac{L}{R}\right)^d,\left(\tfrac{2L}{R}\right)^d\right]\cap\mathbb{N}.
\end{equation}
Give $y_1\in\mathcal{G}$, we apply Corollary~ \ref{Lobe} with $\Omega=\Lambda\supset B(y_1,1)$, $W=V-E$, and $\mathcal{F}=\Ran P$. The hypothesis \eqref{tLb1} follows from \eqref{tmdp5}. We conclude that there exists a vector subspace $\mathcal{F}_{y_1,N}$ of $\Ran P$ and $r_0=r_0(d,p,V^{(1)}_{\infty},V^{(2)}_p,E_0)\in(0,1)$ such that, using \eqref{tmdp11} and \eqref{tmdp9}, we have
\begin{equation}\label{tmdp13}
\dim\mathcal{F}_{y_1,N}\geq\rho L^d-\gamma_dN^{d-1}\geq1,
\end{equation}
and for all $\psi\in\mathcal{F}_{y_1,N}$ we have
\begin{equation}\label{tmdp14}
|\psi(y_1+x)|\leq(C_{d,p,V^{(1)}_{\infty},V^{(2)}_p,E_0}^{N^2}|x|^{N+1}+\varepsilon C_{d,p,V^{(1)}_{\infty},V^{(2)}_p,E_0})\|\psi\|_2\qtx{if}\abs{x}<r_0.
\end{equation}
Picking $y_2\in\mathcal{G}$, $y_2\neq y_1$, and apply Theorem \ref{Lobe} with $\Omega=\Lambda\supset B(y_2,1)$, $W=V-E$, and $\mathcal{F}=\mathcal{F}_{y_1,N}$, we obtain a vector subspace $\mathcal{F}_{y_1,y_2,N}$ of $\mathcal{F}_{y_1,N}$, and hence of $\Ran P$, such that
\begin{equation}\label{tmdp15}
\dim\mathcal{F}_{y_1,y_2,N}\geq\dim\mathcal{F}_{y_1,N}-\gamma_dN^{d-1}\geq\rho L^d-2\gamma_dN^{d-1}\geq1,
\end{equation}
and \eqref{tmdp14} holds for all $\psi\in\mathcal{F}_{y_1,y_2,N}$ also with $y_2$ substituted for $y_1$. Repeating this procedure until we exhaust the sites in $\mathcal{G}$, we conclude that there exists a vector subspace $\mathcal{F}_R$ of $\Ran P$ and $r_0=r_0(d,p,V^{(1)}_{\infty},V^{(2)}_p,E_0)\in(0,1)$, such that
\begin{equation}\label{tmdp16}
\dim\mathcal{F}_R\geq\rho L^d-\left(\tfrac{2L}{R}\right)^d\gamma_dN^{d-1}\geq\tfrac{1}{2}\rho L^d\geq2^{3d}\gamma_d\rho_{ub}\geq1,
\end{equation}
where we used the assumption \eqref{tmdp8}, and for all $\psi\in\mathcal{F}_R$ and $y\in\mathcal{G}$ we have
\begin{equation}\label{tmdp17}
|\psi(y+x)|\leq(C_{d,p,V^{(1)}_{\infty},V^{(2)}_p,E_0}^{N^2}|x|^{N+1}+\varepsilon C_{d,p,V^{(1)}_{\infty},V^{(2)}_p,E_0})\|\psi\|_2\quad\mbox{if}\ x<r_0.
\end{equation}

We let $Q_R$ denote the orthogonal projection onto $\mathcal{F}_R$. Since\ $\tr Q_R=\dim\mathcal{F}_R$, it follows from \eqref{tmdp16} by the argument in \cite[Eqs.~(3.102)-(3.106)]{BKl}
that there  exists $\psi_0=Q_R\psi_0$ with $\|\psi_0\|_2=1$ such that
\begin{equation}\label{tmdp22}
\gamma\rho \le \|\chi_{\Lambda_1}\psi_0\|_2  \le 1,  \qtx{where}\ \gamma=\gamma_{d,p,V^{(1)}_{\infty},V^{(2)}_p,E_0}>0.
\end{equation}

We pick $y_0\in\mathcal{G}$ such that
\begin{equation}\label{tmdp23}
\tfrac{1}{4}<\tfrac{1}{4}R\leq\dist(y_0,\Lambda_1)\leq2\sqrt{d}R,
\end{equation}
which can be done by our construction, and apply Corollary \ref{qucpcor} with $x_0=y_0$, $\Theta=\Lambda_1$, and potential $V-E$; note that
\begin{equation}\label{tmdp24}
\tfrac{R}{4}+\sqrt{d}\leq Q=Q(y_0,\Lambda_1)\leq2\sqrt{d}R+\sqrt{d}\leq3\sqrt{d}R.
\end{equation}
Let $0<\delta<\delta_0:=\min\left\{\frac{1}{2},r_0\right\}$, where $r_0$ is as in \eqref{tmdp17}. It follows from Corollary \ref{qucpcor}, using \eqref{tmdp2}, that
\begin{equation}\label{tmdp25}
\left(\frac{\delta}{3\sqrt{d}R}\right)^{m(1+K^{\frac{2p}{3p-2d}})(R^{\frac{4p-2d}{3p-2d}}-\log\|\psi_0\chi_{\Lambda_1}\|_{2})}\|\psi_0\chi_{\Lambda_1}\|_{2}^2\leq\|\psi_0\chi_{B(y_0,\delta)}\|_{2}^2+\varepsilon^2,
\end{equation}
with a constant $m=m_d>0$ and\ $K=V^{(1)}_{\infty}+V^{(2)}_p+|E|$. Using \eqref{tmdp17} and \eqref{tmdp22}, we get
\begin{align}\label{tmdp26}
&\left(\frac{\delta}{3\sqrt{d}R}\right)^{m(1+K^{\frac{2p}{3p-2d}})(R^{\frac{4p-2d}{3p-2d}}-\log(\gamma p))}(\gamma p)^2
\\&\qquad\qquad\quad\leq C_dC_{d,p,V^{(1)}_{\infty},V^{(2)}_p,E_0}^{N^2}\delta^{2(N+1)+d}+C_{d,p,V^{(1)}_{\infty},V^{(2)}_p,E_0}\varepsilon^2.\nonumber
\end{align}

Since $\rho\geq2R^{-d}$ and $\frac{\delta}{3\sqrt{d}R}<\frac{\delta}{3\sqrt{d}}<1$ by \eqref{tmdp10}, the inequality \eqref{tmdp26} implies the existence of strictly positive constants $\tilde{R}=\tilde{R}_{d,p,V^{(1)}_{\infty},V^{(2)}_p,E_0}$ and $M=M_{d,p,V^{(1)}_{\infty},V^{(2)}_p,E_0}$ such that
\begin{equation}\label{tmdp27}
\left(\tfrac{\delta}{R}\right)^{MR^{\frac{4p-2d}{3p-2d}}}\leq C_{d,p,V^{(1)}_{\infty},V^{(2)}_p,E_0}^{N^2}\delta^{2N}+C_{d,p,V^{(1)}_{\infty},V^{(2)}_p,E_0}\varepsilon^2\quad\mbox{for}\ R\geq\tilde{R}.
\end{equation}
We require
\begin{equation}\label{tmdp28}
R>\widehat{R}=\max\{\tilde{R},\delta_0^{-1}\},
\end{equation}
and choose $\delta$ by (note $C_{d,p,V^{(1)}_{\infty},V^{(2)}_p,E_0}^N\geq1$)
\begin{equation}\label{tmdp29}
\delta=(C_{d,p,V^{(1)}_{\infty},V^{(2)}_p,E_0}^NR)^{-1}<\delta_0,\ \mbox{so}\ \tfrac{\delta}{R}=C_{d,p,V^{(1)}_{\infty},V^{(2)}_p,E_0}^N\delta^2=(C_{d,p,V^{(1)}_{\infty},V^{(2)}_p,E_0}^NR^2)^{-1},
\end{equation}
obtaining
\begin{equation}\label{tmdp30}
\left(\tfrac{\delta}{R}\right)^{MR^{\frac{4p-2d}{3p-2d}}}\leq \left(\tfrac{\delta}{R}\right)^N+C_{d,p,V^{(1)}_{\infty},V^{(2)}_p,E_0}\varepsilon^2.
\end{equation}

We now take $d=2,3$ and take $R$ large enough so that
\begin{equation}\label{tmdp31}
\left(\tfrac{\delta}{R}\right)^N\leq\tfrac{1}{2}\left(\tfrac{\delta}{R}\right)^{MR^{\frac{4p-2d}{3p-2d}}},\quad\mbox{i.e.,}\ (C_{d,p,V^{(1)}_{\infty},V^{(2)}_p,E_0}^NR^2)^{N-MR^{\frac{4p-2d}{3p-2d}}}\geq2.
\end{equation}
To see this, note that $\frac{4p-2d}{3p-2d}<\frac{d}{d-1}$ when $p>\frac{2d}{4-d}$ for $d=2,3$, so
\begin{equation}\label{tmdp32}
MR^{\frac{4p-2d}{3p-2d}}<N=\left\lfloor\left(\frac{\rho}{2^{d+1}\gamma_d}\right)^{\frac{1}{d-1}}R^{\frac{d}{d-1}}\right\rfloor\ \mbox{if}\ \rho>C''_{d,p,V^{(1)}_{\infty},V^{(2)}_p,E_0}R^{\frac{(d-4)p+2d}{3p-2d}},
\end{equation}
and hence
\begin{equation}\label{tmdp33}
(C_{d,p,V^{(1)}_{\infty},V^{(2)}_p,E_0}^NR^2)^{N-MR^{\frac{4p-2d}{3p-2d}}}\!\!\geq4^{N-MR^{\frac{4p-2d}{3p-2d}}}\!\!\geq2\quad\mbox{if}\ \rho>C'''_{d,p,V^{(1)}_{\infty},V^{(2)}_p,E_0}R^{\frac{(d-4)p+2d}{3p-2d}}.
\end{equation}

We now choose $R$ by
\begin{equation}\label{tmdp34}
\rho=c_{d,p,V^{(1)}_{\infty},V^{(2)}_p,E_0}R^{\frac{(d-4)p+2d}{3p-2d}},
\end{equation}
where the constant $c_{d,p,V^{(1)}_{\infty},V^{(2)}_p,E_0}$ is chosen large enough to ensure that, using \eqref{tmdp7},  all the conditions \eqref{tmdp9}, \eqref{tmdp28},  \eqref{tmdp33}, and  \eqref{tmdp31} are satisfied.  It follows from \eqref{tmdp30} and \eqref{tmdp31} that
\begin{align}\label{tmdp35}
&\tfrac{1}{2}\left(\tfrac{\delta}{R}\right)^{MR^{\frac{4p-2d}{3p-2d}}}\leq C_{d,p,V^{(1)}_{\infty},V^{(2)}_p,E_0}\varepsilon^2, \quad\text{that is},
\\&\quad (C_{d,p,V^{(1)}_{\infty},V^{(2)}_p,E_0}^NR^2)^{-MR^{\frac{4p-2d}{3p-2d}}}\leq2C_{d,p,V^{(1)}_{\infty},V^{(2)}_p,E_0}\varepsilon^2.\nonumber
\end{align}
Using  \eqref{tmdp11}, and  \eqref{tmdp34} with a sufficiently large constant $c_{d,p,V^{(1)}_{\infty},V^{(2)}_p,E_0}$, we get from \eqref{tmdp35} that
\begin{equation}\label{tmdp36}
e^{-M'R^{\frac{8p-4d}{3p-2d}}}=e^{-M'R^{\frac{(d-4)p+2d}{(3p-2d)(d-1)}+\frac{d}{d+1}+\frac{8p-4d}{3p-2d}}}\leq C_{d,p,V^{(1)}_{\infty},V^{(2)}_p,E_0}\varepsilon^2,
\end{equation}
where $M'=M'_{d,p,V^{(1)}_{\infty},V^{(2)}_p,E_0}$. Thus
\begin{equation}\label{tmdp37}
\log\frac{1}{\varepsilon}\leq C_{d,p,V^{(1)}_{\infty},V^{(2)}_p,E_0}R^{\frac{8p-4d}{3p-2d}}=\frac{\tilde{C}_{d,p,V^{(1)}_{\infty},V^{(2)}_p,E_0}}{\rho^{\frac{8p-4d}{(4-d)p-2d}}},
\end{equation}
and hence
\begin{equation}\label{tmdp38}
\rho\leq\tilde{C}_{d,p,V^{(1)}_{\infty},V^{(2)}_p,E_0}\left(\log\tfrac{1}{\varepsilon}\right)^{-\frac{(4-d)p-2d}{8p-4d}},
\end{equation}
as long as $L$ is large enough to satisfy \eqref{tmdp9} with the choice of $R$ in \eqref{tmdp34},
namely $L\geq L_{d,p,V^{(1)}_{\infty},V^{(2)}_p,E_0}\left(\log\frac{1}{\varepsilon}\right)^{\frac{3p-2d}{8p-4d}}$.
\end{proof}

\end{document}